\documentclass[10pt, a4paper]{article}
\usepackage{latexsym,amsmath,enumerate,amssymb,amsbsy,amsthm,lscape, hyperref} 
\usepackage {graphicx}
\usepackage{float,color} 
\usepackage{slashbox}

\theoremstyle{plain}
\newtheorem{theorem}{Theorem}[section]
\newtheorem{proposition}[theorem]{Proposition}
\newtheorem{lemma}[theorem]{Lemma}
\newtheorem{corollary}[theorem]{Corollary}

\theoremstyle{definition}

\theoremstyle{remark}
\newtheorem{remark}[theorem]{Remark}
\newtheorem{example}[theorem]{Example}

\numberwithin{equation}{section}

\begin{document}
	\title{On the Enumeration of  Symmetric Tridiagonal Matrices \\ with prescribed Determinant over Commutative \\  Finite Chain Rings}
	\author{Edgar Martinez-Moro\thanks{E. Martinez-Moro is with the Institute of Mathematics, University of Valladolid, Castilla, Spain. Email: edgar.martinez@uva.es}, Neennara Rodnit\thanks{N. Rodnit is with the 
Department of Mathematics, Faculty of Science, Silpakorn University,
Nakhon Pathom 73000, Thailand.
Email: neennararodnit@gmail.com}
,    and Somphong Jitman\thanks{S.~Jitman  is with the 
Department of Mathematics, Faculty of Science, Silpakorn University,
Nakhon Pathom 73000, Thailand. Email: sjitman@gmail.com}$~^,$\thanks{Corresponding Author}}
	\maketitle

	\begin{abstract}
Determinants of structured matrices play a fundamental role in both pure and applied mathematics, with wide-ranging applications in linear algebra, combinatorics, coding theory, and numerical analysis. In this work, the enumeration of symmetric tridiagonal matrices with prescribed determinant over finite fields and over commutative finite chain rings is developed. Using the   recurrent formula for determinants, a recursive form of  the numbers of singular and nonsingular symmetric tridiagonal matrices is derived, after which a uniform counting framework for any fixed determinant value is obtained. Over finite fields, quadratic-character techniques are employed. In odd dimensions, the enumeration is shown to be independent of the chosen nonzero determinant.  Whereas, in even dimensions, it depends only on the quadratic residue/nonresidue class in the fields. For commutative finite chain rings, explicit formulas for nonsingular  case are produced by lifting along the ideal chain and analyzing reduction to the residue field, yielding closed expressions in terms of the nilpotency index and the size of the residue fields.  A layered enumeration has been developed for the study of  symmetric tridiagonal matrices over commutative finite chain rings by stratifying determinants into ideal and punctured layers. Entry-wise reduction to quotient rings expresses each layer through zero-determinant counts on quotients, yielding formulas for prescribed determinants, including quadratic/non-quadratic residue    factors, and a decomposition for the singular ones.



	\end{abstract}

	\section{Introduction}

	Determinants of square matrices are scalar invariants that capture fundamental properties of linear transformations and systems of linear equations. They play a pivotal role in both pure and applied mathematics, with applications ranging across linear algebra, geometry, engineering, computer science, and physics. Consequently, the study of determinants has attracted substantial attention and remains an active area of research (see, \cite{CK2010}, \cite{ST2020}, and references therein).  
	Beyond their theoretical significance, enumeration problems involving determinants present additional mathematical challenges. In \cite{MM1984}, explicit formulas for the number of singular and nonsingular $ n \times n $ matrices over a finite field $ \mathbb{F}_q $ were established. For the ring $ \mathbb{Z}_m $ of integers modulo $ m $, the enumeration of $ n \times n $ matrices with prescribed determinant was initiated in \cite{BM1987} as a natural generalization of the finite field case $ \mathbb{Z}_p $. Subsequently, \cite{LW2007} introduced an alternative and simpler approach to the same problem, and this method was further extended in \cite{CJU2017} to matrices over {commutative finite chain rings (CFCRs)} and {principal ideal rings}, thereby completely determining the number of matrices with a fixed determinant in these settings. More recently, \cite{J2020} considered diagonal matrices over CFCRs with prescribed determinants, with applications to the enumeration of certain circulant matrices, while \cite{JS2024} investigated tridiagonal matrices over CFCRs under the same constraint.

  {Symmetric tridiagonal matrices} constitute a fundamental class of matrices because of their structural simplicity, numerical stability, and computational efficiency (see, e.g., \cite{DP2004}, \cite{CR2013}, \cite{HO1996}, \cite{M1992}). Given their importance, the study of symmetric tridiagonal matrices from both algebraic and enumerative viewpoints is of independent interest.
In this work, we focus on the enumeration of symmetric tridiagonal matrices with prescribed determinants over finite fields and finite chain rings.
 For completeness, we first recall the basic definitions and notations.  
 
	Let $ R $ be a commutative ring with identity and denote by $ U(R) $ its group of units. An element $a \in U(R)$ is called a {\em quadratic residue} if there exists an element $x\in R$ such that  $a=x^2$, and it is said to be  a {\em quadratic nonresidue} otherwise. 
	Let $Q(R)=\{a\in  U(R)  \, :\, \exists x \in R  \text{ such that } a=x^2 \} $ be the set of quadratic residues in   $R$.  Let  $N(R) = (U(R)  )  \setminus  Q(R)$  denote the set of quadratic nonresidues in $R$. A  commutative finite   ring with identity $1 \neq 0$ is referred to as a 
	\textit{ commutative finite   chain ring} (CFCR) if its set of ideals is totally ordered under inclusion; that is, every pair of ideals in $R$ is comparable with respect to set containment. 
	It is well known from \cite{H2001} that every CFCR is, in particular, a principal ideal ring possessing a unique maximal ideal.

	For a positive integer $ n $, an $ n \times n $ matrix over $ R $ is called a \emph{diagonal matrix} if all entries outside the main diagonal are zero. We write
	\[
	{\rm Diag}\begin{pmatrix}
		x_1 & x_{2} & \dots  & x_n
	\end{pmatrix}
	\]
	for the diagonal matrix with diagonal entries $ x_1,x_2,\dots,x_n \in R $.  
	
	A matrix is called \emph{tridiagonal} if its nonzero entries are confined to the main diagonal, the subdiagonal, and the superdiagonal. Explicitly, a tridiagonal matrix over $ R $ has the form
	\begin{align} \label{def:td}
		\left[\begin{array}{ccccccc}
			x_{1} & r_{1}  &        &        &        &        &        \\
			t_{1} & x_{2}  & r_{2}  &        &        &        &        \\
			& t_{2}  & x_{3}  & r_{3}  &        &        &        \\
			&        & \ddots & \ddots & \ddots &        &        \\
			&        &        & t_{n-2}& x_{n-1}& r_{n-1}&        \\
			&        &        &        & t_{n-1}& x_{n}  &        
		\end{array}\right],
	\end{align}
	where $ x_i, r_j, t_j \in R $. A tridiagonal matrix $ A $ is \emph{symmetric} if $ A = A^T $, or equivalently, $ r_i = t_i $ for all $ i=1,\dots,n-1 $.  
	For convenience, we use the notation
	\[
	{\rm TDiag}\begin{pmatrix}
		& r_1 & \dots & r_{n-2} & & r_{n-1} & \\
		x_1 &    & \dots &         & & x_{n-1} & & x_n \\
		& t_1 & \dots & t_{n-2} & & t_{n-1} &
	\end{pmatrix}
	\]
	for a general tridiagonal matrix, and  
	\[
	{\rm STDiag}\begin{pmatrix}
		& r_1 & \dots & r_{n-2} & & r_{n-1} & \\
		x_1 &    & \dots &         & & x_{n-1} & & x_n
	\end{pmatrix}
	\]
	if the matrix is symmetric.

For $n\geq 1$, let    $A_n= 	{\rm STDiag}\begin{pmatrix}
	& r_1 & \dots & r_{n-2} & & r_{n-1} & \\
	x_1 &    & \dots &         & & x_{n-1} & & x_n
\end{pmatrix}$.  	It is well known that, the determinant of an $ n \times n $ symmetric tridiagonal matrix   in 
	$ ST_n(\mathbb{F}_q) $ satisfies the  {recurrent formula}
	\begin{align} \label{eq:det}
		\det ( A_n)  = x_n \det( A_{n-1}) - r_{n-1} ^2 \det(	 A_{n-2})
	\end{align}
for all  $n\geq 3$,  where 	 $	\det ( A_1) =x_1$   and  $	\det ( A_2) =x_1x_2-r_2^2$.  
	
	By a slight abuse of notation, we adopt the convention
$
D_1(R) = ST_1(R) \cong R
$ and $ 
ID_1(R) = IST_1(R) \cong U(R)
$. For $ n \ge 2 $, let $ D_n(R) $  and $ ST_n(R) $ denote the sets of $ n\times n $ diagonal and symmetric tridiagonal matrices over $ R $, respectively. We further write
	\[
	ID_n(R) = \{ A \in D_n(R) \, :\,  \det(A) \in U(R)\} \text{ ~ and ~}
	IST_n(R) = \{ A \in ST_n(R)  \, :\, \det(A) \in U(R)\}.
	\]

	In the case where the ring  $R$ is finite,  we have the following results. 
	\begin{proposition} \label{prop:TnR} Let $n$ be a positive integer and let $ R$ be a  commutative finite    ring with identity.  Then 
		$|D_n(R)|=|R|^n$  and  $|ST_n(R)|=|R|^{2n-1}$.
		
	\end{proposition}

	For a positive integer $n$ and $a\in R$, let \[ST_{n}(R,a)=\{A\in ST_{n}(R)\, :\,  \det ( A)  =a\}.\] 
	
	\begin{example}
		
		Let  
		\[
		A =  {\rm STDiag}\begin{pmatrix}
			&0&&3&&4&&1&\\
			1&&2&&2&&3&&4
		\end{pmatrix} 
		\in ST_5(\mathbb{F}_5).\] Since $\det(A)= 2$, it follows that   $A\in ST_{5}(\mathbb{F}_5,2)$.  
	\end{example}

	\begin{example} For the finite field $\mathbb{F}_5$, the values $|ST_n(\mathbb{F}_5,a)| $  are presented    in Table \ref{T1} for all $a\in \mathbb{F}_5$  and $n\in \{1,2,3,4,5,6,7,8,9,10\}$.

	\begin{table}[!hbt]
		\centering  
\scriptsize
        \begin{tabular}{|c|r|r|r|r|r|r|}
\hline
$n$ & $a=0$ & $a=1$ & $a=2$ & $a=3$ & $a=4$ & $ |IST_n(\mathbb{F}_5)|$ \\
\hline
1  & 1 & 1 & 1 & 1 & 1 & 4 \\
2  & 25 & 30 & 20 & 20 & 30 & 100 \\
3  & 725 & 600 & 600 & 600 & 600 & 2,400 \\
4  & 20,125 & 15,000 & 14,000 & 14,000 & 15,000 & 58,000 \\
5  & 553,125 & 350,000 & 350,000 & 350,000 & 350,000 & 1,400,000 \\
6  & 15,028,125 & 8,500,000 & 8,400,000 & 8,400,000 & 8,500,000 & 33,800,000 \\
7  & 404,703,125 & 204,000,000 & 204,000,000 & 204,000,000 & 204,000,000 & 816,000,000 \\
8  & 11,487,578,125 & 4,762,500,000 & 4,752,500,000 & 4,752,500,000 & 4,762,500,000 & 19,030,000,000 \\
9  & 306,939,453,125 & 114,000,000,000 & 114,000,000,000 & 114,000,000,000 & 114,000,000,000 & 456,000,000,000 \\
10 & 8,148,486,328,125 & 2,731,750,000,000 & 2,730,750,000,000 & 2,730,750,000,000 & 2,731,750,000,000 & 10,925,000,000,000 \\
\hline
\end{tabular}
        \caption{$|ST_n(\mathbb{F}_5,a)|$ for $a\in \mathbb{F}_5$ and $n\in\{1,2,3,4,5,6,7,8,9,10\}$}
		\label{T1}
	\end{table}
 
\end{example}

	From Table \ref{T1}, it can be clearly observed that in the case where $n$ is odd, the cardinality $|ST_n(\mathbb{F}_5,a)|$ remains constant across all units in $\mathbb{F}_5$.  In contrast, when $n$ is even, the value of 
	$|ST_n(\mathbb{F}_5,a)|$  takes the same value for all quadratic residues in 
	$\mathbb{F}_5$ and, separately, the same value for all  quadratic nonresidues in 
	$\mathbb{F}_5$.
	We aim to investigate these observations in detail for arbitrary cases, providing explicit statements and rigorous proofs.

	The primary objectives of this paper are threefold. First, we aim to present a detailed enumeration of the sets $ST_n(R)$ and $IST_n(R)$ in the specific cases where $R$ is taken to be either a finite field or a finite chain ring, thereby providing explicit counting results within these important algebraic structures. Second, we seek to investigate and establish a precise relationship among the sets $ST_n(R,a)$ that arises from and depends essentially on the algebraic properties of the element $a$. Finally, we endeavor to derive and present an explicit formula for $|ST_n(R,a)|$ in a variety of significant cases.

The remainder of this paper is organized as follows. 
Section~\ref{sec3} presents the enumeration of symmetric tridiagonal matrices with prescribed determinants over finite fields. The main results in this section are   explicit counting formulas,   recurrence relations for character sums, and   closed-form expressions in both odd and even dimensional cases.  
Section~\ref{sec4} extends the analysis to  commutative finite   chain rings, which  treats separately the cases of unit and non-unit determinants and obtains enumeration formulas that generalize the finite field results.  
Finally, Section~\ref{sec5} summarizes the main contributions and discusses possible directions for future research.

	%----------------------------------------------------------------
	\section{Enumeration of Symmetric Tridiagonal Matrices with Prescribed Determinant over Finite Fields}
	
	\label{sec3}
	
	This section focuses on the enumeration of symmetric tridiagonal matrices over finite fields with a fixed determinant. First, the   number of $n\times n$ nonsingular  symmetric tridiagonal matrices  and the   number of $n\times n$   singular  symmetric tridiagonal matrices    over finite fields are determined.  Our approach builds upon recursive formulas, structural characterizations, and properties of determinants. Later, a relation among  $|ST_{n}(\mathbb{F}_q,a)|  $  is derived where $a\in U(\mathbb{F}_q)$  together with  an explicit enumeration formula for  the number of $n\times n$ symmetric tridiagonal matrices over $\mathbb{F}_q $ of determinant $a$.

	\subsection{Enumeration of Singular and nonsingular Symmetric Tridiagonal Matrices over Finite Fields} \label{Sec:3.1}

	In this section,   the enumeration of  singular and nonsingular symmetric tridiagonal matrices over finite fields   is presented using a recursive process.
	
	Let $q$ be a prime power and let $n$ be a positive integer.  
	From Proposition~\ref{prop:TnR}, it immediately follows that 
	\[
	|ST_{n}(\mathbb{F}_q)|  = q^{2n-1}.
	\]
	We note that 
	$
	IST_{n}(\mathbb{F}_q) = \{ A \in ST_{n}(\mathbb{F}_q) \, :\, \det(A) \neq 0 \}
	$ is
	the set of all nonsingular $n \times n$ symmetric tridiagonal matrices over the finite field $\mathbb{F}_q$, while the complement $ST_n(\mathbb{F}_q) \setminus IST_n(\mathbb{F}_q)$ is  the set of singular symmetric tridiagonal matrices over $\mathbb{F}_q$.   A recursive formula for $|IST_n(\mathbb{F}_q)|$ is presented in the following theorem.

	\begin{theorem}\label{theIT}
		Let $q$ be a prime power. Then   $|IST_{1}(\mathbb{F}_q)|=q-1$, $|IST_{2}(\mathbb{F}_q)|=q^2(q-1)$ and
		$$|IST_n(\mathbb{F}_q)|=q(q-1)|IST_{n-1}(\mathbb{F}_q)|+q^2(q-1)|IST_{n-2}(\mathbb{F}_q)|$$
		for all $n \ge 3$.
	\end{theorem}
	\begin{proof} 
		Let $n$ be a positive integer. For $n=1$, it follows from the definition of $IST_{1}(\mathbb{F}_q)$. 
		For $n=2$, we have 
		\begin{align*}
			|IST_{2}(\mathbb{F}_q)|
            =|\{(x_1,x_2,r) \in \mathbb{F}_q^3 \, :\, x_1x_2 \ne r^2\}
		%	&=|\mathbb{F}_{q}^3 \backslash \{(r,x,y) \in \mathbb{F}_q^3 \, :\, xy = r^2\}\}\\
			%&=q^3-(q)(q)(1)\\
			=q^2(q-1).
		\end{align*}
		Assume that  $n \ge 3$. Let       \[X={\rm STDiag}\begin{pmatrix}
			&r_1&\dots&r_{n-2}&&r_{n-1}&\\
			x_1&&\dots&&x_{n-1}&&x_n
		\end{pmatrix}
		\in IST_n(\mathbb{F}_q) .\]
		We separate  the proof into  2 cases where  $x_n\ne 0$ and $x_n=0$.
		
		\noindent{\bf Case I}: $x_n \ne 0$.   Let 
		
		\begin{align*}
			Y ={\rm STDiag}\begin{pmatrix}
				&r_1&\dots&r_{n-2}&&0&\\
				x_1&&\dots&&x_{n-1}-(r_{n-1})^2(x_n)^{-1}&&x_n
			\end{pmatrix}.
		\end{align*}
		Then 

		\begin{align*}
			Y&= {\rm TDiag}\begin{pmatrix}
				&0&\dots&0&&-(x_n)^{-1}&\\
				1&&\dots&&1&&1\\
				&0&\dots&0&&0&
			\end{pmatrix} X ~  {\rm TDiag}\begin{pmatrix}
				&0&\dots&0&&0&\\
				1&&\dots&&1&&1\\
				&0&\dots&0&&-(x_n)^{-1}&
			\end{pmatrix} 
		\end{align*}        
		and 
		\begin{align*}\det(X)&=\det(Y)\\
			&=x_n\det\left({\rm STDiag}\begin{pmatrix}
				&r_1&\dots&r_{n-3}&&r_{n-2}&\\
				x_1&&\dots&&x_{n-2}&&x_{n-1}-(r_{n-1})^2(x_n)^{-1}
			\end{pmatrix}\right)\end{align*}
		which implies that  $\det(X) \ne 0 $ if and only if 
		\[\det\left({\rm STDiag}\begin{pmatrix}
			&r_1&\dots&r_{n-3}&&r_{n-2}&\\
			x_1&&\dots&&x_{n-2}&&x_{n-1}-(r_{n-1})^2(x_n)^{-1}
		\end{pmatrix}\right)\ne 0 .\]
		Let 
		$V=\left\{ A \in ST_{n-1}(\mathbb{F}_q)  
		\Bigg| \det\left(A-  {\rm Diag}\begin{pmatrix}
			0&\cdots&0&(r_{n-1})^2(x_n)^{-1}
		\end{pmatrix}\right) \ne 0 \right\}$.  Then $X \in IST_{n}(\mathbb{F}_q)$ if and only if
		\[ {\rm STDiag}\begin{pmatrix}
			&r_1&\dots&r_{n-3}&&r_{n-2}&\\
			x_1&&\dots&&x_{n-2}&&x_{n-1}-(r_{n-1})^2(x_n)^{-1}
		\end{pmatrix}  \in IST_{n-1}(\mathbb{F}_q),\]
		or equivalently, 
		\[{\rm STDiag}\begin{pmatrix}
			&r_1&\dots&r_{n-3}&&r_{n-2}&\\
			x_1&&\dots&&x_{n-2}&&x_{n-1} 
		\end{pmatrix} \in V.\]
        We note that 
		the map $f:V \xrightarrow{} IST_{n-1}(\mathbb{F}_q)$ defined by \[f(A )=A-  {\rm Diag}\begin{pmatrix}
			0&\cdots&0&(r_{n-1})^2(x_n)^{-1}
		\end{pmatrix}\] is a bijection and $|V|=|IST_{n-1}(\mathbb{F}_q)|$. 		Since $x_n \ne 0 $  and $ r_{n-1} $ is an arbitrary element in $ \mathbb{F}_q$, the number of $X$ for which $x_n \ne 0 $ and $ \det(X) \ne 0$ is  
		\begin{align}\label{case1}
			q(q-1)|V|=q(q-1)|IST_{n-1}(\mathbb{F}_q)|.
		\end{align}
		
		\noindent {\bf Case II}: $x_n = 0$. It follows that 
		\[X={\rm STDiag}\begin{pmatrix}
			&r_1&&r_2&\dots&r_{n-2}&&r_{n-1}&\\
			x_1&&x_2&&\dots&&x_{n-1}&&0
		\end{pmatrix}.\]
		By \eqref{eq:det},  we have $\det(X) = (r_{n-1})^2\det(X')$ where 
		\[X'={\rm STDiag}\begin{pmatrix}
			&r_1&&r_2&\dots&r_{n-2}&&r_{n-3}&\\
			x_1&&x_2&&\dots&&x_{n-3}&&x_{n-2}
		\end{pmatrix}.\]
		It follow that $\det(X) \ne 0$ if and only if $\det(X') \ne 0$  and $r_{n-1}\ne 0$.  Since $x_{n-1} $ and $r_{n-2}$ are arbitrary elements in $ \mathbb{F}_q$.
		Therefore, the number of $X$ for  which $x_n = 0$  and $ \det(X) \ne 0$ is 
		\begin{align}\label{case2}
			q^2(q-1)|IST_{n-2}(\mathbb{F}_q)|.
		\end{align}
		
		From (\ref{case1}) and (\ref{case2}),    we conclude that  
		\[|IST_n(\mathbb{F}_q)|=q(q-1)|IST_{n-1}(\mathbb{F}_q)|+q^2(q-1)|IST_{n-2}(\mathbb{F}_q)|\]
		as desired. 
	\end{proof}
	From the recurrence formula of $|IST_n(\mathbb{F}_q)|$ in Theorem 3.1,  an explicit formula of $|IST_n(\mathbb{F}_q)|$ is established in the following theorem.
	\begin{theorem}\label{theITcal}
		Let $n$ be a positive number and let $q$ be a prime power. Then 
		\begin{align*}
			|IST_n(\mathbb{F}_q)|=\beta_1\cdot \lambda_1^n + \beta_2\cdot \lambda_2^n ,
			\end{align*}
			where $\beta_1=\frac{(q+1)+\sqrt{(q-1)(q+3)}}{q\left((q+3)+\sqrt{(q-1)(q+3)}\right)}$,    $\beta_2=\frac{(q+1)-\sqrt{(q-1)(q+3)}}{q\left((q+3)-\sqrt{(q-1)(q+3)}\right)}$,  
			$\lambda_1=\frac{q(q-1) + q\sqrt{(q-1)(q+3)}}{2}$, and $\lambda_2=
			\frac{q(q-1) - q\sqrt{(q-1)(q+3)}}{2}.$
	 
	\end{theorem}
	\begin{proof}
		For a positive integer $n$, let $a_n=|IST_n(\mathbb{F}_q)|$. From Theorem \ref{theIT}, it follows that $a_1=q-1,$ $ a_2 = q^3-q^2$, and $$a_n=q(q-1)a_{n-1}+q^2(q-1)a_{n-2}$$
		for all $n \ge 3$. This equation is a linear homogeneous recurrence equation with characteristic equation
		\begin{align} \label{char}
			a^2=q(q-1)a+q^2(q-1)
		\end{align}
		It follows that
		$$a^2-q(q-1)a-q^2(q-1)=0$$
		and 
		\begin{align*}
			a &= \frac{q(q-1) \pm \sqrt{(q(q-1))^2+4q^2(q-1)}}{2}= \frac{q(q-1) \pm q\sqrt{(q-1)(q+3)}}{2}.
		\end{align*}
		Hence, \eqref{char} has two real distinct roots 
		$\lambda_1=\frac{q(q-1) + q\sqrt{(q-1)(q+3)}}{2}$~ and  ~$\lambda_2=\frac{q(q-1) - q\sqrt{(q-1)(q+3)}}{2}$ which implies that  
		$$a_n = \beta_1\left(\frac{q(q-1) + q\sqrt{(q-1)(q+3)}}{2}\right)^n+\beta_2\left(\frac{q(q-1) - q\sqrt{(q-1)(q+3)}}{2}\right)^n$$
		for some real number $\beta_1$ and $\beta_2.$
		
		Substituting $a_1$ and $a_2$, we have 
		\begin{align}\label{c1}
			p-1 = & \beta_1\frac{q(q-1) + q\sqrt{(q-1)(q+3)}}{2}+\beta_2\frac{q(q-1) - q\sqrt{(q-1)(q+3)}}{2}
		\end{align}
		and
		\begin{align}\label{c2}
			q^2(q-1) = & \beta_1\left(\frac{q(q-1) + q\sqrt{(q-1)(q+3)}}{2}\right)^2+\beta_2\left(\frac{q(q-1) - q\sqrt{(q-1)(q+3)}}{2}\right)^2
		\end{align} 
		Multiplying both sides of (\ref{c1}) by $\frac{q(q-1) - q\sqrt{(q-1)(q+3)}}{2}$, 
		it can be concluded that
		\begin{align}
			\left(\frac{q(q-1) - q\sqrt{(q-1)(q+3)}}{2}\right)(q-1)  
			=& \beta_1\left(\frac{(q(q-1))^2 - (q\sqrt{(q-1)(q+3)})^2}{4}\right)\notag\\
			&+\beta_2\left(\frac{q(q-1) - q\sqrt{(q-1)(q+3)}}{2}\right)^2\notag\\
			=& \beta_1\left(-q^2(q-1)\right)+\beta_2\left(\frac{q(q-1) - q\sqrt{(q-1)(q+3)}}{2}\right)^2. \label{multiply}
		\end{align}
		From (\ref{c2}) $-$ (\ref{multiply}), we have
		\begin{align*}
			q^2(q-1)- &\left(\frac{q(q-1) - q\sqrt{(q-1)(q+3)}}{2}\right)(q-1) 
			= \beta_1\left(\frac{q(q-1) + q\sqrt{(q-1)(q+3)}}{2}\right)^2 + \beta_1\left(q^2(q-1)\right)
		\end{align*}
		which implies that 
		\begin{align*}
			\beta_1 
			%&= \frac{q^2(q-1)- \left(\frac{q(q-1) - q\sqrt{(q-1)(q+3)}}{2}\right)(q-1)}{\left(\frac{q(q-1) + q\sqrt{(q-1)(q+3)}}{2}\right)^2+\left(q^2(q-1)\right)}\\
			%&=\frac{q^2(q-1)- \left(\frac{q(q-1) - q\sqrt{(q-1)(q+3)}}{2}\right)(q-1)}{(1/4)\left((q^2(q-1)^2+2q^2(q-1)\sqrt{(q-1)(q+3)}+q^2(q-1)(q+3)\right)+\left(q^2(q-1)\right)}\\
			%&=\frac{q(q-1)\left[q-\left(\frac{(q-1)-\sqrt{(q-1)(q+3)}}{2}\right)\right]}{(q^2(q-1))\left[(1/4)\left(((q-1)+2\sqrt{(q-1)(q+3)}+(q+3)\right)+1\right]}\\
			%&=\frac{\left(\frac{(q+1)+\sqrt{(q-1)(q+3)}}{2}\right)}{q\left[\frac{\left((2q+2\sqrt{(q-1)(q+3)}+6\right)}{4}\right]}\\
			&=\frac{(q+1)+\sqrt{(q-1)(q+3)}}{q\left((q+3)+\sqrt{(q-1)(q+3)}\right)}.
		\end{align*}
		Similarly, it can be deduced that 
		\begin{align*}
			\beta_2 &= \frac{(q+1)-\sqrt{(q-1)(q+3)}}{q\left((q+3)-\sqrt{(q-1)(q+3)}\right)}.
		\end{align*}
		Therefore,   we have 
		\begin{align*}
	|IST_n(\mathbb{F}_q)|=a_n=\beta_1\cdot \lambda_1^n + \beta_2\cdot \lambda_2^n ,
\end{align*}
as desired.
	\end{proof}
	\begin{example}
		The number of symmetric tridiagonal $n \times n $ matrices over $\mathbb{F}_5$ whose determinant are non-zero can be computed using Theorem \ref{theITcal} as follows: 
	$|IST_{1}(\mathbb{F}_5)| = 4 $,
	$|IST_{2}(\mathbb{F}_5)| = 100 $,
	$|IST_{3}(\mathbb{F}_5)| = 2{,}400 $,
	$|IST_{4}(\mathbb{F}_5)| = 58{,}000 $,
	$|IST_{5}(\mathbb{F}_5)| = 1{,}400{,}000 $,
	$|IST_{6}(\mathbb{F}_5)| = 33{,}800{,}000 $,
	$|IST_{7}(\mathbb{F}_5)| = 816{,}000{,}000 $,
	$|IST_{8}(\mathbb{F}_5)| = 19{,}030{,}000{,}000 $,
	$|IST_{9}(\mathbb{F}_5)| = 456{,}000{,}000{,}000 $, and 
	$|IST_{10}(\mathbb{F}_5)| = 10{,}925{,}000{,}000{,}000$
		which coincide with the  ones in  Table \ref{T1}.
	\end{example}
	Subsequently, $|ST_n(\mathbb{F}_q,0)|$ can be deduced using Theorem \ref{theITcal} which  is presented in the following corollary.
	\begin{corollary}
		Let $n$ be a positive integer and let $q$ be a prime power. Then
		\begin{align*}  |ST_n(\mathbb{F}_q,0)|=&|ST_n(\mathbb{F}_q)|-|IST_n(\mathbb{F}_q)|
			=q^{2n-1}-\beta_1\cdot \lambda_1^n - \beta_2\cdot \lambda_2^n,
		\end{align*}
		where $\beta_1, \beta_2,\lambda_1$, and $\lambda_2$ are defined as in  Theorem \ref{theITcal}.
	\end{corollary}

	\subsection{Enumeration of Symmetric Tridiagonal Matrices with Prescribed Determinant over Finite Fields} \label{Sec:3.2}
	
	In this subsection,   a relation among  $|ST_{n}(\mathbb{F}_q,a)|  $  is  presented  together with  an explicit enumeration formula for  the number of $n\times n$ symmetric tridiagonal matrices over $\mathbb{F}_q $ of determinant $a\in U(\mathbb{F}_q)$.   The results  are given based on    recurrence relations  and the  quadratic character sums.

	The following properties of quadratic residues and quadratic nonresidues in finite fields are  well known (see \cite{W2003}). These are key to study the enumeration of symmetric tridiagonal matrices with prescribed determinant over finite fields. 
    
		\begin{lemma}[{\cite[Theorem 6.18]{W2003}}]  \label{lem:SUF1}	Let $q$ be a  prime power. Then  the following statements hold.
		\begin{enumerate}
						\item If $q$ is even, then    $Q( \mathbb{F}_{q})=U( \mathbb{F}_{q})$ and $|Q( \mathbb{F}_{q})|=q-1 $.
			\item If $q$ is odd, then   $|Q( \mathbb{F}_{q})|=\dfrac{q-1}{2} = |N( \mathbb{F}_{q})|$.
		\end{enumerate} 
	\end{lemma}

	\begin{lemma}[{\cite[Exercise 6.19]{W2003}}]  \label{lem:SUF2}	Let $q$ be an odd prime power. Then    $-1\in  Q( \mathbb{F}_{q})$ if and only if $q\equiv 1 \pmod 4$. 
	\end{lemma}

		\subsubsection{Quadratic Character Sums over Finite Fields}
		
			Let $q$ be an odd prime power and let  $\chi$ be a quadratic character of $\mathbb{F}_q$. Precisely, \[ \chi(a) 
		=\begin{cases}
			0  &\text{ if } a=0,\\
			1	&\text{ if } a\in Q(\mathbb{F}_q), \\
			-1	& \text{ if } a\in N(\mathbb{F}_q).
		\end{cases}
		\]
	For each positive integer $n$ and each prime power $q$,  let 
	\begin{align} \label{eq:sn} S_n(\mathbb{F}_q) =\sum_{a\in U(\mathbb{F}_q) } \chi(a) |ST_n(\mathbb{F}_q,a)|.
	\end{align}
	Then we have the following results.

	\begin{lemma} \label{lem:2.6} Let $q$ be an odd prime power and let $n$ be a positive integer.  Then 
		$S_1(\mathbb{F}_q)=0$, $
		S_2 (\mathbb{F}_q) = (q-1)((2q-1)\chi(-1)-(q-1))
		$, and 
		\[
		S_n(\mathbb{F}_q) = q^2(q-1)\chi(-1)S_{n-2}(\mathbb{F}_q)
		\]
		for all $n\geq 3$.

	\end{lemma}
	\begin{proof}
	From \eqref{eq:det}, 	 the determinant of an $n\times n$ symmetric tridiagonal matrix $A_n \in ST_n(\mathbb{F}_q)$ satisfies the recurrence
	 \begin{align} \label{eq:det2}
		\det ( A_n)  = x_n \det( A_{n-1}) - r_{n-1}^2 \det(	 A_{n-2})
	\end{align}
	for all  $n\geq 3$,  where 	 $	\det ( A_1) =x_1$   and  $	\det ( A_2) =x_1x_2-r_1^2$.

	\noindent {\bf Case 1}: 
		 $n=1$.  The determinant  $\det(A_1)=r_1$ runs uniformly over $\mathbb{F}_q$. Hence,
		\[
		S_1(\mathbb{F}_q) = \sum_{a\in U(\mathbb{F}_q) } \chi(a) |ST_1(\mathbb{F}_q,a)| = \sum_{a\in U(\mathbb{F}_q) } \chi(a) | \{a\}|
		= \sum_{a\in U(\mathbb{F}_q)} \chi(a) = 0 
		\]
		since the nontrivial quadratic character sums to $0$ over $U(\mathbb{F}_q)$.
		
		\noindent {\bf Case 2}:  $n=2$.   In this case, we have $\det(A_2)=x_1x_2 - r_1^2$.  Then
		\begin{align*}
		S_2(\mathbb{F}_q) & =\sum_{a\in U(\mathbb{F}_q) } \chi(a) |ST_2(\mathbb{F}_q,a)|\\
		&= \sum _{a\in U(\mathbb{F}_q) } \chi(a) | \{(x_1,x_2,r_1) \in \mathbb{F}_q^3 \, :\, x_1x_2 - r_1^2=a\}|
	\end{align*}
Since 
	\[| \{(x_1,x_2,r_1) \in \mathbb{F}_q^3 \, :\, x_1x_2 - r_1^2=a\}| = \sum_{t\in \mathbb{F}_q} |\{(x_1,x_2)\in \mathbb{F}_q^2 \, :\,  x_1x_2=t\}|  \cdot | \{(r_1\in \mathbb{F}_q \, :\, t- r_1^2=a\}| ,\]
	we have 
		\begin{align*}
	S_2(\mathbb{F}_q) 
				&= \sum _{a\in U(\mathbb{F}_q) } \chi(a)  \left(\sum_{t\in \mathbb{F}_q} |\{(x_1,x_2)\in \mathbb{F}_q^2 \, :\,  x_1x_2=t\}|  \cdot | \{(r_1\in \mathbb{F}_q \, :\, t- r_1^2=a\}|\right)   \\
				&=  \sum_{t\in \mathbb{F}_q} |\{(x_1,x_2)\in \mathbb{F}_q^2 \, :\,  x_1x_2=t\}|  \cdot \left(\sum _{a\in U(\mathbb{F}_q) } \chi(a)  | \{(r_1\in \mathbb{F}_q \, :\, t- r_1^2=a\}|  \right) .
			\end{align*}
	Since 		
			\[\sum _{a\in U(\mathbb{F}_q) } \chi(a)  | \{(r_1\in \mathbb{F}_q \, :\, t- r_1^2=a\}| =\sum_{r_1\in \mathbb{F}_q}\chi(t-r_1^2), \]
		it follows that 
			\begin{align*}
		S_2(\mathbb{F}_q) 
			 &=
	\sum_{t\in \mathbb{F}_q}  |\{(x_1,x_2)\in \mathbb{F}_q^2 \, :\,  x_1x_2=t\}| \cdot
	\left(\sum_{r_1\in \mathbb{F}_q}\chi(t-r_1^2)\right).
 	\end{align*}
	Using	the classical identity, for odd $q$, we have 
	\[
	\sum_{r_1\in \mathbb{F}_q}\chi(t-r_1^2)=
	\begin{cases}
		q-1 & \text{if } t=0,\\[2pt]
		-1  & \text{if } t\neq 0,
	\end{cases}
	\]
and 
	\[
	 |\{(x_1,x_2)\in \mathbb{F}_q^2 \, :\,  x_1x_2=t\}| =
	\begin{cases}
		2q-1 & \text{if } t=0,\\[2pt]
		q-1  & \text{if } t\neq 0.
	\end{cases}
	\]
	Separating the case where  $t=0$ and where  $t\neq 0$,   we have 
	\[
	S_2(\mathbb{F}_q) = (2q-1)\sum_{r_1\in \mathbb{F}_q}\chi(-r_1^2) 
	+ (q-1)\sum_{t\in U(\mathbb{F}_q)}\sum_{r_1\in\mathbb{F}_q}\chi(t-r_1^2).
	\]
	We note that  $\sum\limits_{r_1\in \mathbb{F}_q}\chi(-r_1^2)=\chi(-1)(q-1)$ and  
	$\sum\limits_{r_1 \in\mathbb{F}_q}\chi(t-r_1^2)=-1$ for all $ t\in U(\mathbb{F}_q)$. Consequently, we have
	\[
	S_2(\mathbb{F}_q) = (2q-1)\chi(-1)(q-1) + (q-1)(q-1)(-1)= (q-1)\left((2q-1)\chi(-1)-(q-1)\right).
	\]

	\noindent {\bf Case 3}:  $n\geq 3$. We note that  
	 \[
	 S_n(\mathbb{F}_q) =\sum_{a\in U(\mathbb{F}_q) } \chi(a) |ST_n(\mathbb{F}_q,a)|
	  =\sum_{A_n\in ST_n(\mathbb{F}_q)} \chi(\det (A_n)).
	 \]
     Based on \eqref{eq:det2}, we consider the following three cases.

	 \noindent {\bf Case 3.1}: $\det( A_{n-1}) \neq 0$. In this case,  $x_n$ and $r_{n-1}$ are arbitrary in  $\mathbb{F}_q$. Then for each 
	  fixed $\det( A_{n-1}) \neq 0$, $x_n \det( A_{n-1}) $   runs uniformly in $\mathbb{F}_q$. Consequently,     the expression
	 \[
		\det ( A_n)  = x_n \det( A_{n-1}) - r_{n-1}^2 \det(	 A_{n-2})
	 \]
	 runs uniformly over all elements of $\mathbb{F}_q$.  
	 Hence,  the $\chi$-sum over all such $ \det( A_{n}) $ vanishes, i.e.,
	 \[
	 \sum_{x_n,r_{n-1} \in \mathbb{F}_q}\chi(\det( A_{n}) )=0.
	 \]

	 \noindent {\bf Case 3.2}: $\det( A_{n-1}) =0$ but $\det( A_{n-2}) \neq 0$.
	 From \eqref{eq:det2}, we have 
	 \[
	 \det( A_{n}) = -r_{n-1}^2 \det( A_{n-2}).
	 \]
	 In this case,  $x_n$ disappears from $ \det( A_{n}) $ which implies that the  sum over all $x_n\in\mathbb{F}_q$ gives a factor $q$.
	 Then
	 \begin{align*}
	  \sum_{x_n,r_{n-1} \in \mathbb{F}_q}\chi(\det( A_{n}) )&= q\sum_{r_{n-1}\in \mathbb{F}_q}\chi(\det( A_{n}))\\
	 &=q \sum_{r_{n-1}\in \mathbb{F}_q}\chi(-r_{n-1}^2 \det( A_{n-2}))\\
     &
	 = q\cdot \chi(-\det( A_{n-2})) \sum_{r_{n-1}\in \mathbb{F}_q}\chi(r_{n-1}^2)\\
     &= q\cdot \chi(-\det( A_{n-2}))(q-1)\\
     &=q\cdot(q-1)\cdot  \chi(-1)\chi(\det( A_{n-2})).
	 \end{align*}
     From \eqref{eq:det2}, 
     \[
\det(A_{n-1})=x_{n-1} \det(A_{n-2}) - r_{n-2}^2\,\det(A_{n-3}).
\]Since $r_{n-2}$ is arbitrary in $\mathbb{F}_q$,  there exists a unique $x_{n-1}$ such that $\det(A_{n-1})=0$. 
  Summing over all  $r_{n-2}$   in $\mathbb{F}_q$ and determinants $\det( A_{n-2})$ of size $(n-2)\times (n-2)$ matrices gives precisely
	 \[
	 q^2(q-1)\chi(-1) S_{n-2}(\mathbb{F}_q).
	 \]

	 \noindent {\bf Case 3} $\det( A_{n-1})=0$ and $\det( A_{n-2})=0$.
	 Then $\det( A_{n})=0$ which implies that  $\chi(\det( A_{n}))=\chi(0)=0$.  
 
	  In conclusion,  we   have the recurrence
	 \[
	 S_n(\mathbb{F}_q) = q^2(q-1)\chi(-1)S_{n-2}(\mathbb{F}_q)
	 \]
	for all  $n\geq 3$. 
	\end{proof}
	
	Using the mathematical induction with the recurrent relation in Lemma \ref{lem:2.6},  the following corollary can be derived immediately. 
	\begin{corollary}  \label{cor:2.7}
		Let $q$ be an odd prime power and $\chi$ the quadratic character on $\mathbb{F}_q$.  
		Let $n $ be a positive integer. Then the following statements holds.  
        \begin{enumerate}
            \item  If 		 $n$ is odd,  then $S_n(\mathbb{F}_q)=0$.
            \item If $n=2m$ is even for some positive integer $m$, then
            \[
		S_{2m}(\mathbb{F}_q) 
		= \left(q^2(q-1)\chi(-1)\right)^{m-1} S_2(\mathbb{F}_q),
		\]
		where 
		\[
		S_2(\mathbb{F}_q) = (q-1)\left((2q-1)\chi(-1)-(q-1)\right).
		\]
        \end{enumerate}
	\end{corollary}

	\begin{corollary} \label{lem:S2-and-rec-mod4}
		Let $q$ be an odd prime power and let $\chi$ be the quadratic character on $\mathbb{F}_q$.
		Then $S_1(\mathbb{F}_q)=0$, and the following case distinctions hold.

	\begin{enumerate}
		\item 	  If $q\equiv 1 \pmod 4$,  then
		\[
		S_2(\mathbb{F}_q)=q(q-1)  \text{ ~ and ~}
		S_n(\mathbb{F}_q)=q^2(q-1)S_{n-2}(\mathbb{F}_q) \text{ ~ for all ~} n\ge 3.
		\]
		In particular, for even $n=2m$,
		\[
		S_{2m}(\mathbb{F}_q)=q^{2m-1}(q-1)^{m}\text{ ~ and ~}
		S_{2m+1}(\mathbb{F}_q)=0 \text{ ~ for all ~} m\ge 1.
		\]
		
	\item  If $q\equiv 3 \pmod 4$, then
		\[
		S_2(\mathbb{F}_q)=-(q-1)(3q-2) \text{ ~ and ~}
		S_n(\mathbb{F}_q)=-q^2(q-1)S_{n-2}(\mathbb{F}_q)\text{ ~ for all ~} n\ge 3.
		\]
		In particular, for even $n=2m$,
		\[
		S_{2m}(\mathbb{F}_q)=(-1)^{m}(3q-2)q^{2m-2}(q-1)^{m}\text{ ~ and ~}
		S_{2m+1}(\mathbb{F}_q)=0 \text{ ~ for all ~} m\ge 1.
		\]
	\end{enumerate} 
	
 \end{corollary}
	
	\begin{proof} The proof is separated into two cases:

		\noindent{\bf Case 1}: $q\equiv 1 \pmod 4$. By Lemma \ref{lem:SUF2},   it follows that  $\chi(-1)=1$. The result can be deduced directly from Lemma \ref{lem:2.6} and  Corollary \ref{cor:2.7}.

		\noindent{\bf Case 2}: $q\equiv 3 \pmod 4$. By Lemma \ref{lem:SUF2},   it follows that  $\chi(-1)=-1$. The result can be deduced directly from Lemma \ref{lem:2.6} and  Corollary \ref{cor:2.7}.
	\end{proof}
	
	\subsubsection{Enumeration of Symmetric Tridiagonal Matrices with Prescribed Determinant}  
	
 This subsection presents      relations among the  numbers
$\,|ST_{n}(\mathbb{F}_q,a)|\,$ as $a$ ranges over $U(\mathbb{F}_q)$  together with  an explicit formula   for
$|ST_{n}(\mathbb{F}_q,a)|$.  
The  results  appear slightly different  according to the characteristic of the fields: when $q$ is odd,
the quadratic character on $\mathbb{F}_q$ yields a nontrivial splitting
by the quadratic residue class, and the   character sum $S_n(\mathbb{F}_q)$  (see Lemma \ref{lem:2.6}) can be
exploited to separate the contributions of  quadratic residues and  quadratic nonresidues. In
contrast, when $q$ is even the squaring map is a permutation of
$U(\mathbb{F}_q)$, the quadratic character is trivial, and all unit
determinants belong to a single class; consequently $S_n(\mathbb{F}_q)$ no
longer provides additional information and the counts
$|ST_{n}(\mathbb{F}_q,a)|$ are uniform in $a\in U(\mathbb{F}_q)$.

For a quadratic residue $a=r^2 \in   Q(\mathbb{F}_q)$,  we note that the map $\alpha_r:ST_n(\mathbb{F}_q,1)\xrightarrow{}ST_n(\mathbb{F}_q,a)$
		defined by 
		
		\begin{align} \label{eq:alpr} \alpha_r( A)= {\rm Diag}\begin{pmatrix}
			1&\dots&1&r 
		\end{pmatrix}  A ~{\rm Diag}\begin{pmatrix}
			1&\dots&1&r 
		\end{pmatrix}\end{align}
		for all  $A\in  ST_n(\mathbb{F}_q,1)$,
	 is a bijection from   $ST_n(\mathbb{F}_q,1)$ onto $ST_n(\mathbb{F}_q,a)$. Moreover,

For an even prime power $q$,    we have $Q(\mathbb{F}_q)=U(\mathbb{F}_q)$ by Lemma \ref{lem:SUF1}.   For each  $a\in U(\mathbb{F}_q)$,    $a=r^2 $ for some $r\in \mathbb{F}_q$ and the map in \eqref{eq:alpr} in a bijection.  Hence,  the next corollary follows immediately.

	\begin{lemma}  \label{cor2.9}	Let $n$ be a positive integer and let $q$ be a prime power.  If $q$ is even, then 
		\[|ST_n(\mathbb{F}_q,a)|= 
		|ST_n(\mathbb{F}_q,1)|  \]
		for all $  a\in U(\mathbb{F}_q)$.   	
	\end{lemma}

		\begin{corollary}  \label{cor2.10}
		Let $n$ be a positive integer and let $a \in U(\mathbb{F}_q)$.  If  $q$ is even, then
		\begin{align*}
			|ST_n(\mathbb{F}_q,a)|
		=\frac{1}{q-1}\left(\beta_1\cdot \lambda_1^n + \beta_2\cdot \lambda_2^n\right)
		\end{align*}
		where $r_1, r_2,\lambda_1$, and $\lambda_2$ are defined as in  Theorem \ref{theITcal}.
	\end{corollary}   
	\begin{proof}
		Assume the $n$ is  odd. 
		Since $IST(\mathbb{F}_q)$ is the disjoint union of those $ST(\mathbb{F}_q,a)$ for which $a \in U( \mathbb{F}_q)$,  we have  
		\[IST_n(\mathbb{F}_q)=\bigcup_{b\in U(\mathbb{F}_q)}^{}(\mathbb{F}_q,b)\]
		and
		\begin{align*}
			\big|IST_n(\mathbb{F}_q)\big|&=\Big|\bigcup_{b\in U(\mathbb{F}_q)}^{}(\mathbb{F}_q,b)\Big|=\sum_{b\in U(\mathbb{F}_q)}^{}\big|ST(\mathbb{F}_q,b)\big|=(q-1)\big|ST(\mathbb{F}_q,1)\big|
		\end{align*}
		By Theorem \ref{theITcal} and  Lemma  \ref{cor2.9},  we have  
		\begin{align*}
			|ST_n(\mathbb{F}_q,a)\big|=&|ST_n(\mathbb{F}_q,1)\big|
			=\frac{|IST(\mathbb{F}_q)\big|}{(q-1)}
				=\frac{1}{q-1}\left(\beta_1\cdot \lambda_1^n + \beta_2\cdot \lambda_2^n\right)
			\end{align*} 
            as desired. 
	\end{proof}

	\begin{theorem} \label{Thm:a=b}
		Let $n$ be a positive integer and let $q$ be a prime power. Let $a\in U(\mathbb{F}_q)$   and $b\in N(\mathbb{F}_q) $.   Then
		\[|ST_n(\mathbb{F}_q,a)|=\begin{cases}
			|ST_n(\mathbb{F}_q,1)| &\text{ if }  a\in Q(\mathbb{F}_q),\\
			|ST_n(\mathbb{F}_q,b)| &\text{ if } a\in N(\mathbb{F}_q) .
		\end{cases}\] 
	\end{theorem}
	\begin{proof}  For the case where  $a \in Q(\mathbb{F}_q)$, the result follows immediately from \eqref{eq:alpr}.  In the case where
		$a  \in N(\mathbb{F}_q)$,  we have  $a^{-1}\in N(\mathbb{F}_q)$ which implies that    $a^{-1}b \in Q(\mathbb{F}_q)$. 
		Precisely,  $a^{-1}b = r^2$ for some $r \in \mathbb{F}_q$.
		Then the map $\beta_r:ST_n(\mathbb{F}_q,a)\xrightarrow{}ST_n(\mathbb{F}_q,b)$
		defined by  
		\[\beta_r( A)= {\rm Diag}\begin{pmatrix}
			1&\dots&1&r 
		\end{pmatrix}  A ~{\rm Diag}\begin{pmatrix}
			1&\dots&1&r 
		\end{pmatrix}\]
		for all  $A\in  ST_n(\mathbb{F}_q,a)$, 
	 is a bijection from   $ST_n(\mathbb{F}_q,a)$ onto $ST_n(\mathbb{F}_q,b)$.    Therefore,       $|ST_n(\mathbb{F}_q,a)|=|ST_n(\mathbb{F}_q,b)|$ as desired.
	\end{proof}

	\begin{theorem}\label{thm:STna} Let $q$ be an odd prime power and let $n$ be a positive integer.  Let  $a\in U(\mathbb{F}_q)$.  Then
		
		\[
		|ST_n(\mathbb{F}_q,a)|=
		\begin{cases}
			\dfrac{|IST_n(\mathbb{F}_q)|+S_n(\mathbb{F}_q)}{q-1}, & \text{if } a  \in Q(\mathbb{F}_q),\\ \\
			\dfrac{|IST_n(\mathbb{F}_q)|-S_n(\mathbb{F}_q)}{q-1}, & \text{if } a \in N(\mathbb{F}_q).
		\end{cases}\]
		
	\end{theorem}
	\begin{proof}  We note that  \[IST_n(\mathbb{F}_q)=\bigcup_{u\in U(\mathbb{F}_q)} ST_n(\mathbb{F}_q,u) =\bigcup_{u\in Q(\mathbb{F}_q)} ST_n(\mathbb{F}_q,u) \cup \bigcup_{u\in N(\mathbb{F}_q)} ST_n(\mathbb{F}_q,u)\]
		is a disjoint union. It follows that 
		\begin{align} 
			|IST_n(\mathbb{F}_q)|&=|\bigcup_{u\in Q(\mathbb{F}_q)} ST_n(\mathbb{F}_q,u) |+| \bigcup_{a\in N(\mathbb{F}_q)} ST_n(\mathbb{F}_q,u)| \notag\\
			&=\sum_{u\in Q(\mathbb{F}_q)} |ST_n(\mathbb{F}_q,u) |+ \sum_{a\in N(\mathbb{F}_q)} |ST_n(\mathbb{F}_q,u)|. \label{eqmain1}
		\end{align}
		From \eqref{eq:sn}, we have 
		\begin{align}
			S_n(\mathbb{F}_q) =\sum_{u\in U(\mathbb{F}_q) } \chi(u) |ST_n(\mathbb{F}_q,u)| = \sum_{u\in Q(\mathbb{F}_q)} |ST_n(\mathbb{F}_q,u) |- \sum_{u\in N(\mathbb{F}_q)} |ST_n(\mathbb{F}_q,u)|\label{eqmain2} .\end{align}				
		Solve the system of linear equations \eqref{eqmain1} and \eqref{eqmain2}, we have 
		\begin{align} \label{eq:Q}
			2 \sum_{u\in Q(\mathbb{F}_q)} |ST_n(\mathbb{F}_q,u) |
			=
			|IST_n(\mathbb{F}_q)|  +S_n(\mathbb{F}_q) 
		\end{align}
		and 
		\begin{align}  \label{eq:N}
			2 \sum_{a\in N(\mathbb{F}_q)} |ST_n(\mathbb{F}_q,a) |
			=
			|IST_n(\mathbb{F}_q)|  -S_n(\mathbb{F}_q) 
		\end{align}			
		For $a\in  Q(\mathbb{F}_q)$,  by Theorem \ref{Thm:a=b},  we have 
		
	\[	\sum_{u\in Q(\mathbb{F}_q)} |ST_n(\mathbb{F}_q,u) |  = \sum_{|Q(\mathbb{F}_q)|}  |ST_n(\mathbb{F}_q,a    ) |= \frac{q-1}{2}  |ST_n(\mathbb{F}_q,a    ) |. \] 
	Together with \eqref{eq:Q},  we have  \[|ST_n(\mathbb{F}_q,a    ) |= \frac{2}{q-1} \sum_{u\in Q(\mathbb{F}_q)} |ST_n(\mathbb{F}_q,u) |=\dfrac{|IST_n(\mathbb{F}_q)|+S_n(\mathbb{F}_q)}{q-1} .\]
			For $a\in  N(\mathbb{F}_q)$,  by Theorem \ref{Thm:a=b},  we have 
	
	\[	\sum_{u\in N(\mathbb{F}_q)} |ST_n(\mathbb{F}_q,u) |  = \sum_{|N(\mathbb{F}_q)|}  |ST_n(\mathbb{F}_q,a    ) |= \frac{q-1}{2}  |ST_n(\mathbb{F}_q,a    ) |. \] 
	Together with \eqref{eq:N},  we have  \[|ST_n(\mathbb{F}_q,a    ) |= \frac{2}{q-1} \sum_{u\in N(\mathbb{F}_q)} |ST_n(\mathbb{F}_q,u) |=\dfrac{|IST_n(\mathbb{F}_q)|-S_n(\mathbb{F}_q)}{q-1} .\]
    This completes the proof. 
	\end{proof}

Since $S_n(\mathbb{F}_q)$ and $|IST_n(\mathbb{F}_q)|$  are determined explicitly in Lemma~\ref{lem:2.6} and Theorem~\ref{theITcal}, the enumeration of symmetric tridiagonal matrices with prescribed determinant over $\mathbb{F}_q$ is thereby completed.

	In the case where $n$ is odd,   $S_n(\mathbb{F}_q)=0$ by Corollary \ref{cor:2.7}.  Hence,  \[|ST_n(\mathbb{F}_q,a)| =|ST_n(\mathbb{F}_q,1)|=\dfrac{|IST_n(\mathbb{F}_q)| }{q-1}\] for all $ a  \in U(\mathbb{F}_q)$ by Theorem \ref{thm:STna}.  Alternatively,    an  explicit link between $|ST_n(\mathbb{F}_q,a)| $ and $|ST_n(\mathbb{F}_q,1)$ is given in the following theorem.

	\begin{theorem}\label{a=1}
		Let $n$ be a positive number and let $q$ be a prime power. If $n$ is odd, then
		\[|ST_n(\mathbb{F}_q,a)|=|ST_n(\mathbb{F}_q,1)|\]
		for all $a \in U( \mathbb{F}_q)$
	\end{theorem}
	\begin{proof}
		Assume  that $n$ is  odd.  Let  $a \in U( \mathbb{F}_q)$.  Then  the map 
		$\varphi_a :ST_n(\mathbb{F}_q,1)\to ST_n(\mathbb{F}_q,a)$
		defined by 
		
		\[\varphi_a( A)= {\rm Diag}\begin{pmatrix}
			a&1&a\dots&1&a
		\end{pmatrix}  A ~{\rm Diag}\begin{pmatrix}
			1&a^{-1}&1\dots&a^{-1}&1 
		\end{pmatrix}\]
		for all  $A\in  ST_n(\mathbb{F}_q,1)$,   is a bijection from   $ST_n(\mathbb{F}_q,1)$ onto $ST_n(\mathbb{F}_q,a)$.    Therefore,       $|ST_n(\mathbb{F}_q,a)|=|ST_n(\mathbb{F}_q,1)|$ as desired.
	\end{proof}

	\begin{corollary}
		Let $n$ be a positive integer and let $a \in U(\mathbb{F}_q)$.  If  $q$ is odd and $n$ is odd, then
		\begin{align*}
			|ST_n(\mathbb{F}_q,a)|
				=\frac{1}{q-1}\left(\beta_1\cdot \lambda_1^n + \beta_2\cdot \lambda_2^n\right)
			\end{align*}
			where $\beta_1, \beta_2,\lambda_1$, and $\lambda_2$ are defined as in  Theorem \ref{theITcal}.
	\end{corollary}   
	\begin{proof}
		 Using the arguments similar to those in the proof of Corollary \ref{cor2.10}.
		While   Theorem \ref{theITcal} and Theorem \ref{a=1} are applied instead of Lemma \ref{cor2.9}
	\end{proof}

	\section{Enumeration of Symmetric Tridiagonal Matrices with Prescribed Determinant over Commutative Finite Chain Rings} \label{sec4}

In this section,  the analysis for the  finite field is extended to  CFCRs. 
The main results generalize those over finite fields: explicit formulas are derived 
for the number of nonsingular symmetric tridiagonal matrices with prescribed determinant, together with
closed expressions that depend only on the nilpotency index $e$ and the  cardinality $q$ of the residue field.

	Let $R$ be a CFCR, and let $\gamma$ denote a fixed generator of its maximal ideal. 
	Then the chain of ideals in $R$ has the canonical form  
	\[
	R \supsetneq \gamma R \supsetneq \gamma^2 R \supsetneq \cdots \supsetneq \gamma^{e-1} R \supsetneq \gamma^e R = \{0\},
	\]
	for some positive integer $e$.  
	The least positive integer $e$ satisfying $\gamma^e = 0$ is called the \emph{nilpotency index} of $R$.  
	Moreover, the quotient ring $R/\gamma R \cong \mathbb{F}_q$ is a finite field of order $q$ for some prime power $q$, and is referred to as the \emph{residue field} of $R$.  By a standard abuse of notation, we write $a + \gamma R \in \mathbb{F}_q$ for the canonical image of any element $a \in R$ under the natural projection $R \to R/\gamma R \cong \mathbb{F}_q$.

	The following lemma, adapted from \cite{H2001} and \cite{HLM2003}, summarizes several basic but useful properties of CFCRs that will be employed in this study.
	
	\begin{lemma}[{\cite{H2001} and \cite{HLM2003}}] \label{lem:propCR}
		Let $R$ be a CFCR with nilpotency index $e$ and residue field $\mathbb{F}_q$ for some prime power $q$.  
		Let $\gamma$ be a generator of the maximal ideal of $R$.  
		Then the following properties hold:
		\begin{enumerate}[$1)$]
			\item For all $0 \leq j \leq e$, the cardinality of the ideal $\gamma^j R$ is given by 
			\[
			|\gamma^j R| = q^{e-j}.
			\]
			\item The group of units of $R$ can be described as
			\[
			U(R) = \{a + b \, :\, a + \gamma R \in U(\mathbb{F}_q), \ b \in \gamma R \}.
			\]
			\item The cardinality of the unit group is 
			\[
			|U(R)| = (q-1) q^{e-1}.
			\]
			\item For each $0 \leq i \leq e$, the quotient ring $R/\gamma^i R$ is itself a CFCR with nilpotency index $i$ and residue field $\mathbb{F}_q$.
		\end{enumerate}
	\end{lemma}
	
Motivated by \cite[Proposition~3]{dHT2024}, we summarize below several foundational properties of the subgroup of  quadratic residues  in  commutative finite   chain rings.

	\begin{lemma} \label{lem:SUR}
		Let $q$ be a prime power. Let $R$ be a CFCR of nilpotency index $e$ and residue field $  \mathbb{F}_{q}$.  Let $\gamma$ be  a generator of the maximal ideal of $R$.       Then the following statements hold.
		\begin{enumerate}
			\item $Q(R)$ is a multiplicative subgroup  of  $U(R)$.
			\item  $Q(R)= \{a+ b\, :\, a+ \gamma R \in   Q(\mathbb{F}_q)  \text{ and } b\in \gamma R\}$.
			\item If $q$ is even, then  $Q(R)=U(R)$ and $|Q(R)|=  (q-1) q^{e-1}$. 
			\item If $q$ is odd, then $ |Q(R)|= |Q(\mathbb{F}_q)||\gamma R|= \dfrac{q-1}{2} q^{e-1}$ and  $ |N(R)| = \dfrac{q-1}{2} q^{e-1}$.
		\end{enumerate}
	\end{lemma}

	\subsection{Enumeration of nonsingular  Symmetric Tridiagonal Matrices with Prescribed Determinant over CFCRs}
		Let $R$ be  a CFCR of nilpotency index $e$ and residue field $\mathbb{F}_q$.      It is well known that $R$ is a disjoint  union of $\{0\}$, $U(R)$, and $Z(R)$.   Let $\gamma$ be a generator of the maximal ideal of   $ R$, By Lemma \ref{lem:propCR}, an element $a\in R$ can be written as $a =\gamma^sb$   for some $0\leq s\leq e$ and     $b\in U(R)$.  Precisely, $a=0$ if $s=e$, $a =\gamma^sb \in  Z(R)$ if $1\leq s\leq e-1$, and $a\in U(R)$ if $s=0$.
	
	\subsubsection{nonsingular  Symmetric Tridiagonal Matrices over CFCRs}

    The enumeration for nonsingular case over a CFCR lifts multiplicatively from the residue field via an entry-wise reduction. 
    Let $\psi:R\to\mathbb{F}_q$ be  the reduction modulo $\gamma R$.  Then the map
	$
	\Psi:ST_n(R)\rightarrow ST_n(\mathbb{F}_q)$ defined by the entry-wise reduction   \begin{align}\label{eq:reduct} A=[a_{ij}] \mapsto  [\psi(a_{ij})] 
	\end{align}
	is a
	surjective additive group homomorphism. Moreover,  the 
 kernel of $\Psi$   consists of  the symmetric tridiagonal matrices all of whose
	$2n-1$ defining entries  are  in $\mathfrak \gamma R$. Hence,
	\begin{equation}\label{eq:ker-size}
		|\ker(\Psi)|=|\gamma R|^{2n-1}=q^{(e-1)(2n-1)},
	\end{equation}
    and hence,  $\Psi$ is a $q^{(e-1)(2n-1)}$-to-one map.

    By applying this reduction, we establish the following results. 

	\begin{theorem}\label{thm:ISTnR}
		Let $R$ be a CFCR with residue field $\mathbb{F}_q$  and nilpotency index $e$.
		Then for all $n\ge 1$,
		\[
		|IST_n(R)|=q^{(e-1)(2n-1)}\cdot |IST_n(\mathbb{F}_q)|,
		\]
		where $|IST_n(\mathbb{F}_q)|$ is given in Theorem~\textnormal{\ref{theITcal}}.
	\end{theorem}
	
	\begin{proof} Based on the homomorphism   $
	\Psi:ST_n(R)\rightarrow ST_n(\mathbb{F}_q)$ defined in \eqref{eq:reduct},   

We note that  $\det(\Psi( A))=\psi({\det(A)})$ which implies that  $\det(A)\in U(R)$ if and only if $\det(\Psi( A))\in U(\mathbb{F}_q)$.
	Therefore, the restriction  $\Psi {\big\vert}_{IST_n(R)} : IST_n(R)\to IST_n(\mathbb{F}_q)$     a surjective map and it is  $|\ker(\Psi)| = q^{(e-1)(2n-1)}$ to one. This completes the proof. 
	\end{proof}

	\begin{remark} The result and proof  of  Theorem \ref{thm:ISTnR} can be naturally generalized to the case of a {\em  commutative finite   local ring}, a  commutative finite     ring  with a unique maximal ideal. Let $\mathfrak{ R}$ be a    commutative finite   local ring with maximal ideal $\mathfrak{M}$ and residue field $\mathbb{F}_q=\mathfrak{ R}/\mathfrak{M}$.  Then  for each positive integer $n$,
		\[
		|IST_n(R)|= |\mathfrak{M}|\cdot |IST_n(\mathbb{F}_q)|,
		\]
		where $|IST_n(\mathbb{F}_q)|$ is given in Theorem~\textnormal{\ref{theITcal}}.
	\end{remark}

For a CFCR with residue field $\mathbb{F}_q$ of  even characteristic,    we have $Q(R)=U(R)$ by Lemma \ref{lem:SUR}.    Similar to Corollary \ref{cor2.9},  the next corollary  can be deduced immediately.

	\begin{lemma}  \label{lem3.5}Let $R$ be  a CFCR of nilpotency  index $e$ and residue field $\mathbb{F}_q$, and let  $n$ be a positive integer.   
	   If $q$ is even, then 
		\[|ST_n(R,a)|= 
		|ST_n(R,1)|  \]
		for all $  a\in U(R)$.   	
	\end{lemma}

	By Theorem \ref{thm:ISTnR} and Lemma \ref{lem3.5}, we have the following enumeration. 
		\begin{corollary}  \label{cor2.10}
		Let $R$ be  a CFCR of nilpotency  index $e$ and residue field $\mathbb{F}_q$,  and let  $n$ be a positive integer.  If  $q$ is even, then
		\begin{align*}
			|ST_n(R,a)|
		=\frac{|IST_n(R)|}{(q-1)q^{e-1}}.  
		\end{align*} 
	\end{corollary}

	\begin{theorem} \label{lem:sq-class-inv}
	
	Let $R$ be  a CFCR of nilpotency  index $e$ and let  $n$ be a positive integer.   
	Let $a\in U(R)$   and $b\in N(R) $.   Then
	\[|ST_n(R,a\gamma^s )|=\begin{cases}
		|ST_n(R,\gamma^s )| &\text{ if }  a\in Q(R),\\
		|ST_n(R,\gamma^s b)| &\text{ if } a\in N(R) 
	\end{cases}\] 
	for all $s\in \{0,1,\dots,e\}$. 
	\end{theorem}

	\begin{proof}  The proof can be extended naturally  from  the case of finite fields in Theorem \ref{Thm:a=b}, where   
 Lemma~\ref{lem:SUR}   is applied instead of 
Lemma~\ref{lem:SUF1}.
	\end{proof}

 Let $R$ be a   CFCR with nilpotency index $e$,  maximal ideal
 $\gamma R$, and residue field $R/\gamma R\cong\mathbb{F}_q$.
  Let $\chi$ be the quadratic character on $\mathbb{F}_q$.  Using the reduction   $\psi:R\to\mathbb{F}_q$,  let  $\eta_R:=\chi\circ\psi$.   By Lemma~\ref{lem:SUR}, we have  $u\in Q(R)$   if and only if $\psi(u) \in Q(\mathbb{F}_q)$.    Consequently,

 \[ 
\eta_R(a)=
 \begin{cases}
 	1 & \text{ if } a\in Q(R),\\
    -1 & \text{ if } a\in N(R),\\
 	0, & \text{ if } a\notin U(R).
 \end{cases}
 \]
 
 For each integer $n\ge1$,  let 
 \[
 S_n(R) := \sum_{a\in U(R)}\eta_R(a) |ST_n(R,a)|
  = \sum_{A\in ST_n(R)}\eta_R(\det (A)).
 \]
 Then we have the following recursive formula for  $ S_n(R)$.
 \begin{lemma} \label{lem:lift-Sn} Let $q$ be an odd prime power and let $n$ be  a positive integer.  
 	Let $R$ be a CFCR with  nilpotency index $e$ and residue field $\mathbb{F}_q$. Then 
 	\[
 	S_n(R) = q^{(e-1)(2n-1)} S_n(\mathbb{F}_q),
 	\] where  $S_n(\mathbb{F}_q)$ is determined in Lemma \ref{lem:2.6}.
 \end{lemma}
 
 \begin{proof}
 Based on $\Psi$ defined in \eqref{eq:reduct},  we write $\Psi^{-1}(B)$ for the inverse image of $B\in ST_n(\mathbb{F}_q) $.  Since $\Psi$ is    a $q^{(e-1)(2n-1)}$-to-one map,  $|\Psi^{-1}(B)|=q^{(e-1)(2n-1)}$ for all $B\in ST_n(\mathbb{F}_q) $.  It follows that 
 	\[
 	\begin{aligned}
 		S_n(R)
 		&=\sum_{A\in ST_n(R)}\eta_R(\det (A))\\
        &
 		=\sum_{\Psi(A)\in ST_n(\mathbb{F}_q)} \sum_{A\in \Psi^{-1}(\Psi(A))}\eta_R(\det (A))\\
        &=\sum_{\Psi(A)\in ST_n(\mathbb{F}_q)}  |\Psi^{-1}(\Psi(A))|  \chi(\det\Psi(A))\\
 		&=\sum_{\Psi(A)\in ST_n(\mathbb{F}_q)} q^{(e-1)(2n-1)} \chi(\det\Psi(A))\\
        &
 		= q^{(e-1)(2n-1)} \sum_{\Psi(A)\in ST_n(\mathbb{F}_q)} \chi(\det\Psi(A))\\
 		&= q^{(e-1)(2n-1)} S_n(\mathbb{F}_q)
 	\end{aligned}
 	\]
 as desired.
 \end{proof}

By considering the residue class of $q$ modulo $4$, the following specialization of
Lemma~\ref{lem:lift-Sn} and Lemma~\ref{lem:2.6} is obtained.
 \begin{corollary} \label{cor:SnR-qmod4}
Let $q$ be an odd prime power and let $m$ be  a positive integer.  
 	Let $R$ be a CFCR with  nilpotency index $e$ and residue field $\mathbb{F}_q$. Then,
\[
S_{2m-1}(R)=0 \quad 
\text{ and } \quad 
S_{2m}(R)=q^{(4m-1)(e-1)}\bigl(q^2(q-1)\chi(-1)\bigr)^{m-1}S_2(\mathbb{F}_q),
\]
where $S_2(\mathbb{F}_q)=(q-1)\big((2q-1)\chi(-1)-(q-1)\big)$. 
Equivalently, the closed forms are:
\begin{enumerate}
\item If $q\equiv 1\pmod 4$, then
\[
S_{2m}(R)
= q^{e(4m-1)-2m} (q-1)^m.
\]
\item If $q\equiv 3\pmod 4$, then
\[
S_{2m}(R) 
= (-1)^m (3q-2) q^{e(4m-1)-2m-1} (q-1)^m.
\]
\end{enumerate}
\end{corollary}
\begin{proof} From Lemma~\ref{lem:lift-Sn},  we have 
\begin{equation}\label{eq:lift}
 S_{2m-1}(R)=0 \quad 
\text{ and }  \quad S_{2m}(R)=q^{(4m-1)(e-1)}S_{2m}(\mathbb{F}_q).
\end{equation}
By Corollary~\ref{cor:2.7},  we have 
\[ 
S_2(\mathbb{F}_q)=(q-1)\bigl((2q-1)\chi(-1)-(q-1)\bigr)
\]
and
\begin{equation}\label{eq:field-closed}
S_{2m}(\mathbb{F}_q)=\bigl(q^2(q-1)\chi(-1)\bigr)^{m-1}S_2(\mathbb{F}_q).
\end{equation}
Combining \eqref{eq:lift} and \eqref{eq:field-closed} yields
\begin{equation}\label{eq:SnR-even-general}
S_{2m}(R)=q^{(4m-1)(e-1)}\bigl(q^2(q-1)\chi(-1)\bigr)^{m-1}S_2(\mathbb{F}_q).
\end{equation}
 The remaining parts follow from Corollary \ref{lem:S2-and-rec-mod4}.
\end{proof}

	\begin{theorem}\label{thm:split-formula-ring} Let $q$ be an odd prime power and let $n$ be a positive integer. 
		Let $R$ be a CFCR with residue field $\mathbb{F}_q$ and nilpotency index $e$.   Let  $a\in U(R)$.  Then 
		\[
		|ST_n(R,a)|=
		\begin{cases}
			\dfrac{|IST_n(R)|+S_n(R)}{(q-1)q^{e-1}} & \text{ if }a\in Q(R),\\[10pt]
			\dfrac{|IST_n(R)|-S_n(R)}{(q-1)q^{e-1}}  & \text{ if } a\notin Q(R).
		\end{cases}
		\]
		In particular, if $n$ is odd, then $S_n(R)=0$ and
		\[
		|ST_n(R,a)|=\dfrac{|IST_n(R)|}{(q-1)q^{e-1}},
		\] 
        for all $a\in U(R)$.
	\end{theorem}
	
	\begin{proof} The proof can be generalized naturally  from  Theorem \ref{thm:STna} while Theorem \ref{lem:sq-class-inv} with $s=0$ is applied instead of Theorem  \ref{Thm:a=b}. 
		 
	\end{proof}

	\subsection{Enumeration of Singular Symmetric Tridiagonal Matrices with Prescribed Determinant over Commutative Finite Chain Rings}

In this subsection,   a layered technique  for counting symmetric tridiagonal matrices with a
prescribed determinant is developed by considering  the {determinant
layer} in which $\det(A)$ lies: the ideal layers $\gamma^sR$ ($0\le s\le e$) and their
\emph{punctured} counterparts $\gamma^sU(R)=\gamma^sR\setminus\gamma^{s+1}R$ when $R$ is a   CFCR with maximal ideal $\gamma R$ and 
nilpotency index $e$.  The entry-wise reduction
$\psi: R\to R/\gamma^tR$ yields   the layer counting over $R$ in
terms of  $|ST_n(R/\gamma^tR,0)|$ on the quotient.  This leads to explicit
formulas for
\[
\bigl|\{A:\det(A)\in\gamma^sR\}\bigr| \quad  \text{ and } \quad
\bigl|\{A:\det(A)\in\gamma^sU(R)\}\bigr|.
\]
Further, a complete determination of $ |ST_n(R,a)| $ is obtained for  $a=0$  and 
$ a=\gamma^{s}b \in R $ with $1\le s\le e-1$ and $b\in U(R)$.
By Theorem~\ref{lem:sq-class-inv}, the enumeration bifurcates according to the square
class of the unit factor $b$: the case $b\in Q(R)$   and the case
$b\in N(R)$   lead to distinct enumeration.
The two situations will therefore be presented separately and they are sufficient to given in terms of  $ |ST_n(R,\gamma^{s})| $  and $ |ST_n(R,\gamma^{s} b)| $ with $b\in N(R)$.

\subsubsection{Layered Enumeration for Determinants}

We note that for $t\in\{1,\dots,e\}$, the quotient $R/\gamma^t R$ has cardinality $|R/\gamma^t R|=q^t$,
and the set of $n\times n$ symmetric tridiagonal matrices over $R/\gamma^tR$ has size
$|ST_n(R/\gamma^tR)|=q^{t(2n-1)}$. The  layered enumeration  for the determinants is given in the next theorem.

\begin{theorem}\label{thm:ideal-and-punctured-layer}
Let $R$ be a commutative finite chain ring with maximal ideal $\gamma R$, nilpotency index $e$,
and residue field $R/\gamma R\cong\mathbb{F}_q$. Let  $n$ be a positive integer. Then the following statements hold.

 \begin{enumerate}
\item 
For each $s\in\{1,\dots,e\}$,
\[
\bigl|\{\,A\in ST_n(R)\;:\;\det(A)\in \gamma^s R\,\}\bigr|
=q^{(e-s)(2n-1)}\bigl|ST_n(R/\gamma^sR,\,0)\bigr|.
\]
In particular, if $s=0$, then  $\gamma^0R=R$ and 
\[
\bigl|\{\,A\in ST_n(R)\;:\;\det(A)\in R\,\}\bigr|=\bigl|ST_n(R)\bigr|=q^{e(2n-1)}.
\]

\item 
For  each $s\in\{1,\dots,e-1\}$,
\begin{align*}
\bigl|\{\,A\in ST_n(R)\;:\;\det(A)\in \gamma^s U(R)\,\}\bigr|
=&q^{(e-s)(2n-1)}\bigl|ST_n(R/\gamma^sR,\,0)\bigr|
\\
&-q^{(e-s-1)(2n-1)}\bigl|ST_n(R/\gamma^{s+1}R,\,0)\bigr|.
\end{align*}
In particular, if  $s=0$,  then $\gamma^0U(R)=U(R)$ and
\[
\bigl|\{\,A\in ST_n(R)\;:\;\det(A)\in U(R)\,\}\bigr|
=|IST_n(R)|=q^{(e-1)(2n-1)}|IST_n(\mathbb{F}_q)|.
\]
     
 \end{enumerate}
\end{theorem}

\begin{proof}
Let  $s\in\{1,\dots,e\}$. We note the  the entry-wise reduction map
$\Psi_s:ST_n(R)\to ST_n(R/\gamma^sR)$  is surjective group homomorphism with 
$|\ker(\Psi_s)|=|\gamma^sR|^{\,2n-1}=q^{(e-s)(2n-1)}$.  Moreover, 
$\det(A)\in\gamma^sR$  if and only if $\det(\Psi_s(A))=0\in R/\gamma^sR$. Therefore,
\[
\bigl|\{A\in ST_n(R):\det(A)\in\gamma^sR\}\bigr|
=|\ker(\Psi_s)|\cdot\bigl|ST_n(R/\gamma^sR,0)\bigr|
=q^{(e-s)(2n-1)}\bigl|ST_n(R/\gamma^sR,0)\bigr|.
\]
The first statement is proved.

For the second part, the punctured layer
\[
\{A:\det(A)\in\gamma^sU(R)\}
=\{A:\det(A)\in\gamma^sR\}\setminus\{A:\det(A)\in\gamma^{s+1}R\}.
\]
  yields
\[
q^{(e-s)(2n-1)}\bigl|ST_n(R/\gamma^sR,0)\bigr|
-q^{(e-s-1)(2n-1)}\bigl|ST_n(R/\gamma^{s+1}R,0)\bigr|
\]
as desired.
\end{proof}

\subsubsection{Explicit formula for $\bigl|ST_n(R,\gamma^s)\bigr|$} 

Here, we present the enumeration formula for $\bigl|ST_n(R,\gamma^s b )\bigr|$, where $b\in Q(R)$ and $s\in \{1, 2,\dot,e-1\}$.  By Theorem~\ref{lem:sq-class-inv},  $\bigl|ST_n(R,\gamma^s b )\bigr|=\bigl|ST_n(R,\gamma^s )\bigr|$ which is  given in the following theorem.   

\begin{theorem}\label{thm:exact-gammas}
Let $R$ be a CFCR with maximal ideal $\gamma R$, nilpotency index $e$,  and  residue field
$\mathbb{F}_q$.  Let $n$ be a positive integer and let  $s\in \{1, 2,\dots,e-1\}$. Then 
\begin{equation}\label{eq:gamma-s-count}
\bigl|ST_n(R,\gamma^s)\bigr|
=\frac{q^{\,2(e-s-1)(n-1)}}{q-1}
\Bigl(q^{\,2n-1}\bigl|ST_n(R/\gamma^sR,0)\bigr|
      -\bigl|ST_n(R/\gamma^{s+1}R,0)\bigr|\Bigr).
\end{equation}
\end{theorem}

\begin{proof}First, consider the entry-wise reduction
$\mu:ST_n(R/\gamma^{s+1}R)\to ST_n(R/\gamma^{s}R)$.
We note the the kernel  of $\mu$ consists  of symmetric tridiagonal matrices with all entries in $\gamma^sR/\gamma^{s+1}R$ which implies that  $|\ker(\mu)|=q^{\,2n-1}$.
Counting the pre-images of the zero-determinants $ST_n(R/\gamma^sR,0)$,  we have 
\[
q^{\,2n-1}\bigl|ST_n(R/\gamma^sR,0)\bigr|
=\bigl|ST_n(R/\gamma^{s+1}R,0)\bigr|
+(q-1)\,\bigl|ST_n(R/\gamma^{s+1}R,\gamma^s+\gamma^{s+1}R)\bigr|.
\]
Rearranging, it follows that 
\begin{equation}\label{eq:quotient-step}
\bigl|ST_n(R/\gamma^{s+1}R,\gamma^s+\gamma^{s+1}R)\bigr|
=\frac{1}{q-1}\Bigl(q^{\,2n-1}\bigl|ST_n(R/\gamma^sR,0)\bigr|
                   -\bigl|ST_n(R/\gamma^{s+1}R,0)\bigr|\Bigr).
\end{equation}
We note that 
 entry-wise reduction
$\Psi_{s+1}:ST_n(R)\to ST_n(R/\gamma^{s+1}R)$  is a surjective group homomorphism with  $|\ker(\Psi_{s+1})|=|\gamma^{s+1}R|^{\,2n-1}
= q^{(e-s-1)(2n-1)}$.  Then the restriction map \[\Psi_{s+1}\big\vert_{ST_n(R,\gamma^s)}:   ST_n(R,\gamma^s) \to ST_n(R/\gamma^{s+1}R,\gamma^s+\gamma^{s+1}R)\] is surjective. 
  For a fixed
$
B\in ST_n(R/\gamma^{s+1}R,\ \gamma^s+\gamma^{s+1}R)$, every lift $A \in ST_n(R) $ of $B$ has determinant of the form
$
\det(A)=\gamma^s u$,  
where $u\in    1+\gamma ( R/ \gamma^{e-s-1} R)$ as embedded in $U(R)$.   We note that $\det(A)=\gamma^s$ if and only if  $u= 1$. As the $2n-1$ entries in the
kernel $\gamma^{s+1}R$ are arbitrary, the resulting unit factor $u$ runs {uniformly} over
 $1+\gamma ( R/ \gamma^{e-s-1} R)$   whose cardinality is $q^{\,e-s-1}$. Therefore, among the $|\ker(\Psi_{s+1})|=q^{(e-s-1)(2n-1)}$ lifts of $B$, exactly a
proportion $1/q^{\,e-s-1}$ achieve $u=1$,  or equivalently,  $\det(A)=\gamma^s$.
Hence, the number of
lifts of $B$ with {exact} determinant $\gamma^s$  is
$|\ker(\Psi_{s+1})|/q^{\,e-s-1}
= q^{(e-s-1)(2n-1)-(e-s-1)}=q^{\,2(e-s-1)(n-1)}$.
In particular, the restriction
$
\Psi_{s+1}\big\vert_{ST_n(R,\gamma^s)}$
is a $q^{\,2(e-s-1)(n-1)}$-to-one surjective map.  Consequently, 
\begin{equation}\label{eq:lift-factor}
\bigl|ST_n(R,\gamma^s)\bigr|
= q^{\,2(e-s-1)(n-1)}\ \bigl|ST_n(R/\gamma^{s+1}R,\gamma^s+\gamma^{s+1}R)\bigr|.
\end{equation}
Combining \eqref{eq:quotient-step} and \eqref{eq:lift-factor}, it yields the result\eqref{eq:gamma-s-count}.
\end{proof}
In the case where  $q$ is even,  we have $U(R)=Q(R)$ by Lemma \ref{lem:SUR}. Hence,   
 
\[\bigl|ST_n(R,\gamma^s a)\bigr|= \bigl|ST_n(R,\gamma^s )\bigr| \]
for all $a\in U(R)$ as given in Theorem  \ref{thm:exact-gammas}.

\subsubsection{Explicit formula  for $\bigl|ST_n(R,\gamma^s b)\bigr|$ with $b\in N(R)$}

Here, we focus on  the number $\bigl|ST_n(R,\gamma^s b)\bigr|$ in the case where $b\in N(R)$ which actually occurs only odd prime power $q$. In this case, $\bigl|ST_n(R,\gamma^s b)\bigr|$ shares the value among the elements $b\in N(R)$ by  Theorem~\ref{lem:sq-class-inv}.

\begin{theorem}\label{thm:nonsq} Let $q$ be an odd prime power and let $n$ be a positive integer. 
Let $R$ be a CFCR with maximal ideal $\gamma R$, nilpotency index $e$, and  residue field
$\mathbb{F}_q$.  Let  $u \in N(R)$ and let  $s\in \{1, 2,\dots,e-1\}$. Then 
 
\begin{align*}  
\bigl|ST_n(R,u\,\gamma^s)\bigr|
=&\frac{2}{(q-1)\,q^{\,e-1}}
\Bigl(q^{(e-s)(2n-1)}\bigl|ST_n(R/\gamma^sR,0)\bigr|
     -q^{(e-s-1)(2n-1)}\bigl|ST_n(R/\gamma^{s+1}R,0)\bigr|\Bigr)\\&
 -\bigl|ST_n(R,\gamma^s)\bigr|,
\end{align*}
where $\bigl|ST_n(R,\gamma^s)\bigr|$ is given by \eqref{eq:gamma-s-count}. 
\end{theorem}

\begin{proof}   
  From Theorem \ref{thm:ideal-and-punctured-layer}, it follows that 
\begin{align}
\bigl|\{A:\ \det(A)\in\gamma^sR\setminus\gamma^{s+1}R\}\bigr|&=\bigl|\{\,A\in ST_n(R)\;:\;\det(A)\in \gamma^s U(R)\,\}\bigr|\notag \\ &
=q^{(e-s)(2n-1)}\bigl|ST_n(R/\gamma^sR,0)\bigr|
 -q^{(e-s-1)(2n-1)}\bigl|ST_n(R/\gamma^{s+1}R,0)\bigr|. \label{eq:layer-total}
\end{align}
This layer decomposes as a disjoint union of the fibers $\{ST_n(R,a\gamma^s):a\in U(R)\}$. Since $q$ is odd,  we have 
$|Q(R)|=|N(R)|=\frac{q-1}{2}q^{\,e-1}$ by Lemma \ref{lem:SUR}.
Alternatively, we have 
\begin{align*}
\bigl|\{A:\ \det(A)\in\gamma^sR\setminus\gamma^{s+1}R\}\bigr|&=
|Q(R)|\cdot \bigl|ST_n(R,\gamma^s)\bigr|
    +|N(R)|\cdot \bigl|ST_n(R,u\gamma^s)\bigr|\\
    &
=\frac{(q-1)q^{\,e-1}}{2}\Bigl(\bigl|ST_n(R,\gamma^s)\bigr|+\bigl|ST_n(R,u\gamma^s)\bigr|\Bigr).
\end{align*}
Rearranging,  
\[
\bigl|ST_n(R,u\gamma^s)\bigr|
=\frac{2\bigl|\{A:\ \det(A)\in\gamma^sR\setminus\gamma^{s+1}R\}\bigr|}{(q-1)\,q^{\,e-1}}-\bigl|ST_n(R,\gamma^s)\bigr|.
\]
Together with   \eqref{eq:layer-total}  and \eqref{eq:gamma-s-count}, the result follows. 
\end{proof}

\subsubsection{The Zero-Determinant   $\bigl|ST_n(R,0)\bigr|$}

The resulting identity expresses $\bigl|ST_n(R,0)\bigr|$ as the total
number of symmetric tridiagonal matrices minus the invertible ones and minus the
contributions of all punctured layers. Each punctured layer is written
as a difference of zero–determinant counts on adjacent quotients
$R/\gamma^sR$ and $R/\gamma^{s+1}R$. In particular, once the numbers
$\bigl|ST_n(R/\gamma^tR,0)\bigr|$ are known (or computed recursively by
the triangular scheme developed earlier), the formula below yields
$\bigl|ST_n(R,0)\bigr|$ outright. Obtaining closed forms for
$\bigl|ST_n(R/\gamma^tR,0)\bigr|$ for general $n,t$ remains  an
open component in this approach.

\begin{theorem}\label{thm:zero-fiber}
Let $R$ be a CFCR with maximal ideal $\gamma R$, nilpotency index $e$, residue field
$\mathbb{F}_q$, and let $n\ge 1$.
Then
\begin{align}
\bigl|ST_n(R,0)\bigr|
=&q^{e(2n-1)}-|IST_n(R)|  \notag\\
&-\sum_{s=1}^{e-1}
\Bigl(q^{(e-s)(2n-1)}\bigl|ST_n(R/\gamma^sR,0)\bigr|
     -q^{(e-s-1)(2n-1)}\bigl|ST_n(R/\gamma^{s+1}R,0)\bigr|\Bigr). \label{eq:zero-by-layers}
\end{align}
\end{theorem}

\begin{proof} 
The  set  $ST_n(R)$  can be viewed as the following disjoint  union:
\[
ST_n(R)=IST_n(R)\ \dot\cup\ \Bigl(\bigcup_{s=1}^{e-1}\ \bigcup_{a\in U(R)} ST_n(R,a\gamma^s)\Bigr)
\ \dot\cup\ ST_n(R,0).
\]
This implies that 
\begin{align}\label{eq0}
\bigl|ST_n(R,0)\bigr|
=|ST_n(R)|-|IST_n(R)|-\sum_{s=1}^{e-1}\ \sum_{a\in U(R)} \bigl|ST_n(R,a\gamma^s)\bigr|.
\end{align}
From Theorem \ref{thm:ideal-and-punctured-layer}, it follows that 
\begin{align}
\sum_{a\in U(R)} \bigl|ST_n(R,a\gamma^s)\bigr|&=\bigl|\{\,A\in ST_n(R)\;:\;\det(A)\in \gamma^s U(R)\,\}\bigr|\notag \\ &
=q^{(e-s)(2n-1)}\bigl|ST_n(R/\gamma^sR,0)\bigr|
 -q^{(e-s-1)(2n-1)}\bigl|ST_n(R/\gamma^{s+1}R,0)\bigr|. \label{eq:layer-totaaal}
\end{align}
Together with  $|ST_n(R)|=q^{e(2n-1)}$  and \eqref{eq0},  the proof is completed.
\end{proof}

	\section{Conclusion and Remarks}
	\label{sec5}

	This work develops an explicit enumeration of  nonsingular symmetric tridiagonal matrices with prescribed determinant over finite fields and  CFCRs.
	Over finite fields, we  present the  recurrence for tridiagonal determinants to derive and solve a closed linear recurrence for the number of nonsingular symmetric tridiagonal matrices. Using quadratic character techniques, we further analyze how the counts partition by the determinant value: a sharp parity phenomenon emerges—when the dimension is odd, the number of matrices is independent of the chosen nonzero determinant, while in even dimensions the count depends only on the quadratic residue class of the determinant, with a refinement governed by the field’s congruence modulo four. These results yield compact formulas and transparent structural insight into determinant distributions within this classical matrix family.
	
	Over CFCRs, we obtain explicit formulas for the nonsingular case by reducing to the residue field and lifting along the ideal chain. In particular, the enumeration factors cleanly through the  nilpotency index and the size of the residue field.  For singular  symmetric tridiagonal matrices over a CFCR,  the enumeration has been made by sorting them according to the level  of their determinant inside the ideal chain. Each level has two versions: the whole ideal and the same set with its next, smaller ideal removed. A simple reduction map from the ring to its quotients lets us convert counts over the ring into counts of matrices whose determinant is zero over smaller rings. This yields explicit formulas for determinants equal to a fixed power of the maximal ideal. What remains open is a closed form for the zero-determinant enumeration  on each quotient level.

	%---------------------------------------------------------------

\end{document}